\documentclass[12pt,onecolumn,journal]{IEEEtran}
\usepackage{cite}
\usepackage{amsmath,amssymb,amsfonts}
\usepackage{mathtools}
\usepackage{algorithmic}
\usepackage{graphicx}
\usepackage{subfig}
\usepackage{algorithm}
\usepackage{textcomp}
\usepackage{setspace}
\usepackage{flushend}
\def\BibTeX{{\rm B\kern-.05em{\sc i\kern-.025em b}\kern-.08em
    T\kern-.1667em\lower.7ex\hbox{E}\kern-.125emX}}

\newtheorem{theorem}{Theorem}[section]

\newtheorem{definition}[theorem]{Definition}

\newtheorem{condition}[theorem]{Condition}
\newtheorem{problem}[theorem]{Problem}
\newtheorem{example}[theorem]{Example}
\newtheorem{remark}[theorem]{Remark}
\newtheorem{procedure}[theorem]{Offline AFD Procedure}

\begin{document}
\title{Active Fault Diagnosis for a Class of  Nonlinear Uncertain Systems: \\A Distributionally Robust Approach}
\author{Ioannis Tzortzis, and Marios M. Polycarpou
\thanks{``This work is funded by the European Union's Horizon 2020 research and  innovation programme under grant agreement No 739551 (KIOS CoE) and from  the Republic of Cyprus through the Directorate General for European  Programmes, Coordination and Development.'' }
\thanks{``This is an extended and revised version of a preliminary conference paper that was presented in ECC'19 \cite{Tzortzis:2019}.'' }
\thanks{The authors are with KIOS Research and Innovation Center of Excellence and the Department of Electrical and Computer Engineering, University of Cyprus (e-mails: \{tzortzis.ioannis,\ mpolycar\}@ucy.ac.cy).}}

\maketitle

\begin{abstract}
This work is devoted to the development of a distributionally robust active fault diagnosis approach for a class of nonlinear systems, which takes into account any  ambiguity in distribution information of the uncertain model parameters. More specifically, a new approach is presented using the total variation distance metric as an information constraint, and as a measure for the separation of multiple models based on the similarity of their output probability density functions. 
A practical aspect of the proposed approach is that different levels of ambiguity may be assigned to the models pertaining to the different fault scenarios.
The main feature of the proposed solution is that it is expressed in terms of distribution's first and second moments, and hence, can be applied to alternative distributions other than normal. In addition, necessary and sufficient conditions of optimality are derived, and an offline active fault diagnosis procedure is provided to exemplify the implementation of the proposed scheme. The effectiveness of the proposed distributionally robust  approach is demonstrated through an application to a three-tank benchmark system under multiple fault scenarios. 
\end{abstract}


\section{Introduction}
Active Fault Diagnosis (AFD) aims at designing  optimal separating input signals that enhance the capability for detection and isolation of faults in dynamical systems. Classical methods for AFD are developed under the assumption that the probability distribution associated with uncertain parameters present in the models is available and accurate. In practice, however,  distribution information  is constructed based on modeling considerations, prior knowledge about the parameters, observation collected data, etc. If poor and/or partial and incomplete distribution information is considered then model-based systems  will be sensitive to  errors, and consequently, may compromise the performance of the optimal separating input signals designed for effective fault diagnosability. In practice, this problem becomes particularly acute given  that most modern systems are too complex to model accurately in their entirety. Hence, robustness against ambiguity in distribution information of uncertain model parameters is one of the fundamental and most challenging issues in designing and analyzing  AFD schemes.

During the last few years, several robustness approaches have been developed in the area of AFD to deal with measurement noise, process disturbances, uncertain model parameters, model errors, model nonlinearities, etc. Based on the various models of uncertainty,  AFD approaches can be classified into two main categories, namely, deterministic approaches and probabilistic approaches. Deterministic approaches typically treat uncertainties as unknown bounded signals with known bounds. Good overviews of  deterministic approaches, both for linear and nonlinear systems, can be found in \cite{Nikoukhah:2006,Andjelkovic:2008,Scott:2013,Raimondo:2013,Paulson:2014,MARSEGLIA2017,Yang21,XU2021}. Deterministic AFD approaches with bounded uncertainties, under a robust formulation, have also received  attention in \cite{Yao20,Nikoukhah:2000,Ashari:2012,Chen:2012,Streif:2013,Zhang:2002}. Probabilistic approaches for AFD are based on the uncertainties being  described by stochastic processes with known probability distribution functions \cite{Kerestecioglu:1993,Blackmore:2008}. Related work on probabilistic AFD which deals with the design of auxiliary separating input signals  of
multiple nonlinear models related to faultless and faulty system operations can also be found in \cite{Mesbah:2014,Mesbah:2014b,Scott:2013,Poulsen:2008}. 
The rationale behind this paper, in contrast to the existing literature, is to address an issue of pivotal importance in AFD, that is, the effect of ambiguous distribution information and its impact on the performance of the optimal separating input signals designed for effective fault diagnosability. 

Motivated by the above discussion, and by the fact that in many practical AFD applications distribution information of uncertain model parameters is known ambiguously, in this paper we investigate  a \emph{distributionally robust} AFD  approach using Total Variation (TV) distance metric: (i) as a tool for codifying the level of ambiguity, and (ii) as a measure for the separation of multiple models (which  are used to describe the faultless as well as the faulty system). To handle ambiguous distribution information, the proposed approach minimizes a worst-case payoff over an ambiguity set of  probability density functions (PDFs). Model output PDFs  are obtained through the evolution of the dynamical models under parametric uncertainties, utilizing certain estimation methods, such as the maximum likelihood  estimation. Although such an estimation may be imprecise, mainly because of the ambiguity in the distribution information, it is often our best guess and contains valuable information about the stochastic nature of the underlying uncertainty. A natural way to reduce the influence of ambiguity on the performance of the optimal separating input signals is to consider some level of perturbation or deviation from the estimated  PDFs, henceforth, called the \emph{nominal} PDFs. Specifically, we assume that the \emph{true} PDFs of model outputs are not completely known but contained in a pre-specified ambiguity set, defined by the TV distance metric,  and are within a  distance $R\in[0,1]$ from the nominal PDF. The AFD is then  performed with respect to the true  PDFs, that is, by designing an optimal  input signal that enables the separation of the support of the true PDFs of model outputs pertaining to the different fault scenarios.

The proposed distributionally robust AFD approach is quite general and can handle ambiguity issues related to parametric uncertainty in various model-based applications \cite{CHENG2021107353,Cheng9340542,Stojanovic,Tzortzis9157997}. The following important features  are distinguished. First, by appropriately adjusting the TV distance parameter $R\in[0,1]$ we can control the degree of conservatism of the AFD problem. In particular, for $R=0$ the distributionally robust AFD problem reduces to the standard stochastic AFD problem without ambiguity. On the other hand, as the TV distance parameter  increases, then highly ambiguous scenarios are taken into account. Second,  we are allowed to assign different levels of ambiguity in distribution information, a feature that is ignored in the extant literature.  This side-steps the issue of having to assign the same TV distance between  models  which correspond to different scenarios as could be the case, for example, between the faultless model and a novel faulty model.
The main contributions of this work are summarized as follows.
\begin{enumerate}
\item An upper bound which can be evaluated in closed form is provided for the evaluation of the common area between an arbitrary number of probability density functions. This problem reformulation benefits since it is expressed in terms of distribution's first and second moments, and hence, can be applied to alternative distributions other than normal;
\item The problem of ambiguous distribution information is addressed by deriving the necessary and sufficient Karush-Kuhn-Tucker conditions of optimality;
\item An  algorithm, together with an offline  procedure, are provided to solve the AFD problem and to exemplify the implementation of the proposed  scheme. Through the proposed procedure a piecewise constant input signal emerges which when applied to the true system the separation of multiple models at different time instances is enabled;
\item To illustrate the effect of ambiguous distribution information, an in-depth discussion and comparison of the developed results are provided through an application to a three-tank system under multiple fault scenarios.
\end{enumerate}

The remainder of this paper is organized as follows. In Section \ref{section.problem.form}  the robust AFD problem is formulated and the total variation distance metric is introduced. In Section  \ref{sec.robust.sol}, the solution of the robust AFD problem is derived, a robust AFD algorithm is provided, and an AFD scheme is proposed. In Section \ref{sec.example}, a three-tank system under multiple fault scenarios is considered as an application of the proposed methodology. Finally, concluding remarks  and future research directions are given in Section \ref{sec.conclusion}.

\section{Problem Formulation}\label{section.problem.form}
\subsection{Model description}
We consider continuous-time nonlinear models of the following form
\begin{equation}\label{eq1}
m^{[j]}:\begin{cases}
               \dot{x}^{[j]}(t)=g^{[j]}\left(x^{[j]}(t),u(t),\theta^{[j]}\right)\\
               y^{[j]}(t)=h^{[j]}\left(x^{[j]}(t),u(t),\theta^{[j]},v(t)\right)
            \end{cases}
\end{equation}
where the superscript $j\in I:=\{0,1,\dots,n_f\}$ denotes the model index and its associated set of variables, $n_f$ is the number of possible faults, with one faultless model $m^{[0]}$, and $n_f$ faulty models. The functions $g^{[j]}$ and $h^{[j]}$, denote the nonlinear system dynamics and the model output, respectively. The variables $x^{[j]}\in \mathcal{X}\subset \mathbb{R}^{n_x}$, $u\in \mathcal{U}\subset \mathbb{R}^{n_u}$, $\theta^{[j]}\in \mathbb{R}^{n_\theta}$ and $y^{[j]}\in\mathbb{R}^{n_y}$, denote the system states, input, time-invariant (uncertain) parameters, and outputs available for fault diagnosis, respectively.  In addition, the time-varying variable $v(t)$ denotes the measurement noise, and represents the reliability of the measurements. 

A fault is defined as an unpermitted deviation of at least one characteristic property of the system from the acceptable, usual, nominal condition \cite{Isermann:2006}. For active fault diagnosis, dynamical models given by \eqref{eq1}  are used to describe the behaviour of the faultless and the faulty system under uncertain model parameters.
In this paper, it is assumed that:
\begin{enumerate}
\item[(i)] system faults are known apriori and are separable, 
\item[(ii)]  the input signal $u(t)\in \mathcal{U}$ is  the same for all models and can attain values from an admissible set $\mathcal{U}$, and 
\item[(iii)] the initial condition $x^{[j]}(0)=x_0^{[j]}$, or its distribution, is known for all $\forall j\in I$.
\end{enumerate}
To simplify the notation,  the dimensions of all models, $m^{[j]}$, $j\in I$, and their associated set of variables are kept the same. 
 Due to the ambiguous distribution information of uncertain model parameters, considered in this paper, the robust AFD problem is investigated under the following condition.
\begin{condition}\label{cond1}
The uncertain model parameters $\theta^{[j]}$, $j\in I$, consist of independent  distributed random variables which are not fully known but with pre-specified nominal PDFs, $f(\theta^{[j]})$, $j\in I$.
\end{condition}

By the evolution of dynamical models under  Condition \ref{cond1},  the nominal distribution information provided for uncertain model parameters $\theta^{[j]}$, $j\in I$, leads to probabilistic distributed outputs with corresponding nominal PDFs, denoted by $f(y^{[j]}(t))$, $j\in I$. Going a step further, we assume that the true PDFs of model outputs, denoted by $\tilde{f}(y^{[j]}(t))$, $j\in I$, are not fully known but contained in an ambiguity set centered around the  nominal PDFs $f(y^{[j]}(t))$,  with corresponding radius $R^{[j]}\in [0,1]$, $j\in I$. 
\begin{remark}
Small values of the TV distance parameter $R^{[\cdot]}$ (i.e., values  close to zero) imply that the true and the nominal PDFs are close to each other. However, as  the value of $R^{[\cdot]}$ increases, then the ambiguity set increases, which in turn implies that highly uncertain  scenarios are taken into account.\footnote{Intuitively, the TV distance parameter $R^{[\cdot]}$ can be seen as an alternative to the use of multiple separate models for the representation of uncertainty.}
\end{remark}
The key point here is that if the support of the true output PDFs overlap then measurements may make the different fault scenarios non-isolable \cite{Mesbah:2014}. Hence, the aim of this paper is, in addition to designing an optimal separating  input signal which when injected to the true system existing faults can be effectively detected and isolated, to restrict also the influences of ambiguous distribution information about uncertain model parameters on the overall system.

Based on the above discussion, the  robust AFD problem under ambiguous distribution information is stated as follows.

\subsubsection*{Robust AFD Problem}
 Design a distributionally robust input signal $u(t)\in \mathcal{U}$,  with respect to ambiguous distribution information of uncertain model parameters, so that when applied to the true system, actual measurements can be associated with high probability with only one fault scenario $m^{[j]}$, $j\in I$.
   
Hence, the proposed AFD problem deals with the design of a distributionally robust separating input signal which when applied to the true system complete fault isolation can be achieved.    We note that, in general, the designed optimal separating input signal can be applied either on its own or as an auxiliary to the nominal/stabilizing control input signal (i.e., the system can have a nominal/stabilizing input signal under normal operation). However, in that case, one must investigate possible side effects of the additional auxiliary separating input signal to the system behaviour. Ideally, the designer would like to keep the faultless model with auxiliary input as close as possible to the faultless model without auxiliary input, while at the same time, try to maximize the  distance between the faultless model and the faulty models  with auxiliary input signal. Although this case is not addressed in the current work  it will be subject to future investigation.

   Next, the TV distance metric is introduced.

\subsection{Total variation distance metric}
The robust AFD problem will be formulated in terms of the TV distance metric for which the definition follows.  

\begin{definition}[TV Distance Metric]  Let $(\mathcal{X},\mathcal{B}(\mathcal{X}))$ denote an arbitrary measurable space and $\mathcal{M}_1(\mathcal{X})$ the set of probability measures on $\mathcal{X}$. The TV distance between two probability measures is a function $d_{TV}:\mathcal{M}_1(\mathcal{X})\times\mathcal{M}_1(\mathcal{X})\longmapsto [0,\infty)$, defined by \cite{dunford}
\begin{equation}\label{tvdfn}
d_{TV}(p,q)=\sup_{P\in \mathcal{P}(\mathcal{X})}\sum_{F_i\in P}\lvert p(F_i)-q(F_i)\rvert
\end{equation}
where $p,q\in \mathcal{M}_1(\mathcal{X})$ and $\mathcal{P}(\mathcal{X})$ denotes the collection of all finite partitions $P=\{F_1,F_2,\dots,F_{n_p}\}$ on $\mathcal{X}$.
\end{definition}

Let $p$ and $q$ be absolutely continuous with respect to the Lebesgue measure so that  $f_{p}(x)\triangleq \frac{dp}{dx}(x)$, $f_{q}(x)\triangleq \frac{dq}{dx}(x)$ (i.e., $f_p(\cdot)$, $f_q(\cdot)$ are the probability density functions of $p(\cdot)$ and $q(\cdot)$, respectively). Then, the TV distance metric is given by \cite{pinsker}
\begin{equation}\label{TV.eq}
d_{TV}(p,q)=\frac{1}{2}\int_{\mathcal{X}}\lvert f_{p}(x)-f_{q}(x)\rvert dx.
\end{equation}
Two PDFs  $f_{p}$, $f_{q}$ with a TV distance equal to $1-a$ share a common area of size $a$. Thus the more they overlap, the closer they are,  and hence, the greater is their common area. 

In this work, the use of TV distance is twofold, in particular, it is used: (1) as a measure for the evaluation of the common area between true model output PDFs which correspond to multiple fault scenarios, due to its natural and intuitive interpretation applicable to the problem at hand, and, (2) as an information constraint to model ambiguity based on the nominal and the true model output PDFs. The emphasis on TV distance to model ambiguity is motivated by its generality, since it applies to any arbitrary distribution induced either by linear or nonlinear models.

In the literature  of AFD, different optimality criteria have been proposed and used for the separation of multiple models. A related formulation, which deals with the separation of multiple models, through the Bhattacharyya coefficient, is developed in \cite{Blackmore:2008}, where the authors provide an upper bound on the Bayes probability of model selection error leading to quadratic cost functions. Ambiguity sets constructed based on alternative distances such as Kullback-Leibler divergence, Hellinger distance, Wasserstein metric can be found for example in \cite{doi:10.1287/mnsc.1120.1641,MohajerinEsfahani2018}. An elaborated review of some of the most important metrics on probability measures and the relationships between them can be found in \cite{gibbs}. In addition, various applications based on the TV distance metric and its relation to other distance metrics can also be found in \cite{ctlthem:2013}.

\section{Robust Active Fault Diagnosis}\label{sec.robust.sol}
\subsection{Optimal input design under ambiguous distribution information}\label{subsec.robust.sol.partial}
In this section, we formulate the input design problem under ambiguous distribution information of uncertain model parameters.  For convenience,  we denote by $f^{[j]}\triangleq {f}(y^{[j]}(t))$, $\tilde{f}^{[j]}\triangleq \tilde{f}(y^{[j]}(t))$, and by $d_{TV}^{t_m}(\tilde{f}^{[i]},\tilde{f}^{[j]})$ the TV distance (in the sense of \eqref{TV.eq}) between the PDFs $\tilde{f}^{[i]}$ and $\tilde{f}^{[j]}$  at time instant $t=t_m\in[0,T]$. The specific problem is stated as follows.

\begin{problem}\label{problem2}
Find an optimal input signal $u(t)\in \mathcal{U}$, $t\in [0,T]$, and optimal PDFs $\tilde{f}^{[j]}$, $\forall j\in I$, which solve the following minimax optimization problem:
\begin{align}\label{opt.cost.problem2}
\min_{u(t)}\max_{\tilde{f}^{[j]},\forall j\in I}&\ \sum_{i=0}^{n_f-1}\sum_{j=i+1}^{n_f}\left(1-d_{TV}^{t_m}\left(\tilde{f}^{[i]},\tilde{f}^{[j]}\right)\right)\\
\mbox{subject to:}&\   \dot{x}^{[j]}(t)=g^{[j]}\left(x^{[j]}(t),u(t),\theta^{[j]}\right),\ x^{[j]}(0){=}x^{[j]}_0 \nonumber\\
&\  y^{[j]}(t)=h^{[j]}\left(x^{[j]}(t),u(t),\theta^{[j]},v(t)\right)\nonumber\\
&\  d_{TV}^{t_m}\left(\tilde{f}^{[j]},f^{[j]}\right)\leq R^{[j]}, \quad  \forall j{\in} I\nonumber\\
&\  u(t)\in \mathcal{U},\quad x(t)\in \mathcal{X}\nonumber
\end{align}
where $t_m\in[0,T]$ denotes the time instant over which the common area between the output PDFs is evaluated, and $R^{[j]}$, $\forall j{\in} I$, denotes the TV distance parameter which is given and belongs to $[0,1]$.
\end{problem}

Our objective is to design a distributionally robust input signal $u(t)\in \mathcal{U}$ that minimizes the worst-case cost given ambiguous information about the PDFs of model outputs.  In particular, to select an input signal $u(t)\in\mathcal{U}$  that enables the separation of the true output PDFs, $\tilde{f}^{[j]}$, $j\in I$, by minimizing their common area at time instant $t_m\in [0,T]$, subject to hard input and state constraints. Hard input and state constraints guarantee that the system states and inputs remain bounded for all $t\geq 0$,  (i.e., there exists a region $D\subseteq \mathcal{X} \times \mathcal{U}$ such that $x(t),u(t)\in D$, $\forall t\geq 0$). Note that, although in Problem \ref{problem2} the separation of multiple models is achieved at a single time instant $t_m$ over the finite time horizon $[0,T]$, in general, the separation at several time instances it is also possible. Toward  this, in Section \ref{sec.algorithm} we propose an offline AFD procedure in which a piecewise constant input signal is designed so that when applied to the true system over $[0,T]$ enables the separation of multiple models at different time instances $t_{m_k}\in [0,T]$, $k=1,2,\dots$. 

Alternatively, Problem \ref{problem2} can be viewed as a minimax dynamic game with two players opposing each others actions. Specifically, in this minimax game formulation the objective of player I is to choose an input signal $u(t)\in \mathcal{U}$ to minimize the common area between the model output true PDFs pertaining to the different fault scenarios, while the objective of player II is to choose the model output true PDFs, over all possible PDFs within the TV distance inequality constraint, so that the common area  is maximized.
\begin{remark}
The inequality constraint, or alternatively, the ambiguity set in Problem \ref{problem2} also accounts for scenarios in which $R^{[i]}\neq R^{[j]}$, $i,j\in I$. This means that we are allowed to assign different levels of ambiguity between the nominal and the true model output PDFs for different fault scenarios. This feature is of practical importance with applications in cases where we don't have the same amount of confidence  in model output PDFs pertaining to the different fault scenarios. 
\end{remark}

The desired nominal PDFs, $f^{[j]}$, $j\in I$, of model outputs are obtained by performing Monte Carlo simulations of nonlinear model \eqref{eq1}, using the nominal distribution information of uncertain model parameters. For presentation purposes, throughout the paper it is assumed that  the nominal probability density functions, $f^{[j]}$, $j\in I$, are obtained by normal approximations of the resulting probability histograms using Maximum Likelihood (ML) estimation \cite{Theodoridis2008}. That is, given $y_i^{[j]}\in \mathbb{R}^{n_y}$, $i=1,2,\dots,N$, $j\in I$, the ML estimates of the unknown  mean value and variance that define the nominal PDFs are given by
\begin{equation}
\mu_j=\frac{1}{N}\sum_{i=1}^Ny_i^{[j]}\quad \mbox{and}\quad \sigma^2_j =\frac{1}{N}\sum_{i=1}^N(y_i^{[j]}-\mu_j)^2.
\end{equation}


Although special attention is given to the case of normal approximations, we note that  if the approximation of the probability histograms by normal distributions is poor and not accurate enough, one may choose alternative approximating distributions that provide  better fitness results to the resulting histograms.\footnote{An alternative approach for addressing the problem of poor and not accurate approximations may be obtained by working directly on the resulting probability histograms utilizing equation \eqref{tvdfn}.} Indeed, since the main results of this work will be derived in terms of the distribution's first and second moments, alternative distributions may also be used (i.e., see Example \ref{toyexample}). It is worth noting however  that, the designer always needs to balance the trade-off between precision and computational complexity. In addition, since by definition, normal distributions can take any value from $-\infty$ to $+\infty$, when the value is far away from the mean then it is assumed that its probability is practically negligible.\footnote{This is adjusted by appropriately selecting the interval of confidence.}

Given that the nominal output PDFs, $f^{[j]}$, $\forall j\in I$, are obtained by normal approximations, then it follows that the true output PDFs are given by
\begin{equation}
\tilde{f}^{[j]}\triangleq \tilde{f}^{[j]}(y|\tilde{\mu}_j,\tilde{\sigma}_j)=\frac{1}{\tilde{\sigma}_j\sqrt{2\pi}}\exp\left(-(y-\tilde{\mu}_j)^2/2\tilde{\sigma}_j^2\right)
\end{equation}
where $\tilde{\mu}_j$ and $\tilde{\sigma}_j$ denote the mean and standard deviation of model $m^{[j]}$, $j\in I$, respectively.  Notice that,  inner optimization of Problem \ref{problem2} needs to be solved for all possible mean $\tilde{\mu}_j$ and standard deviation $\tilde{\sigma}_j$, $\forall j\in I$, of the true output PDFs. Hence, the inner optimization of Problem \ref{problem2}  becomes
\begin{align}\label{max.problem}
&\max_{\tilde{f}^{[j]},\forall j\in I}\sum_{i=0}^{n_f-1}\sum_{j=i+1}^{n_f}\left(1-d_{TV}^{t_m}\left(\tilde{f}^{[i]},\tilde{f}^{[j]}\right)\right)\nonumber\\
&=\max_{\substack{\tilde{\mu}_j,\tilde{\sigma}_j \\ \forall j\in I}}\sum_{i=0}^{n_f-1}\ \sum_{j=i+1}^{n_f}\left(1-d_{TV}^{t_m}\left(\tilde{f}^{[i]}(y|\tilde{\mu}_i,\tilde{\sigma}_i),\tilde{f}^{[j]}(y|\tilde{\mu}_j,\tilde{\sigma}_j)\right)\right)
\end{align}
with TV distance constraint
\begin{equation}\label{TVconst}
 d_{TV}^{t_m}\left(\tilde{f}^{[j]}(y|\tilde{\mu}_j,\tilde{\sigma}_j),f^{[j]}(y|\mu_j,\sigma_j)\right)\leq R^{[j]}.
\end{equation}
When clear from the context, the conditioning of the PDF's $\tilde{f}^{[j]},f^{[j]}$, on the means, $\tilde{\mu}_j,\mu_j$, and standard deviations, $\tilde{\sigma}_j,\sigma_j$, respectively, will be dropped. 
Next, we provide  an upper bound on the common area between a finite number of normal PDFs in terms of TV distance. Subsequently, we use this upper bound to solve the inner optimization of Problem \ref{problem2}.

\begin{theorem}\label{thm1}
The common area in terms of TV distance metric between a finite number of normally distributed model output PDFs, denoted by $\tilde{f}^{[i]}(y|\tilde{\mu}_i,\tilde{\sigma}_i)$, $i\in I\triangleq \{0,1,\dots,n_f\}$,  is upper bounded by
\begin{equation}
\sum_{i=0}^{n_f-1}\sum_{j=i+1}^{n_f}\left(1-d_{TV}\left(\tilde{f}^{[i]}(y|\tilde{\mu}_i,\tilde{\sigma}_i),\tilde{f}^{[j]}(y|\tilde{\mu}_j,\tilde{\sigma}_j)\right)\right)
\leq \sum_{i=0}^{n_f-1}\sum_{j=i+1}^{n_f}\left(\frac{1}{2}+\frac{2\tilde{\mu}_i\tilde{\mu}_j+\tilde{\sigma}_i^2+\tilde{\sigma}_j^2}{2(\tilde{\mu}_i^2+\tilde{\mu}_j^2+\tilde{\sigma}_i^2+\tilde{\sigma}_j^2)}\right).\label{upperboundgeneral}
\end{equation}
\end{theorem}

\begin{proof}
Choose a pair $(k,l)\in I$ of normally distributed output PDFs. Then
\begin{align*}
&\left(\int_{\mathcal{Y}} y\left(\tilde{f}^{[k]}(y|\tilde{\mu}_k,\tilde{\sigma}_k)-\tilde{f}^{[l]}(y|\tilde{\mu}_l,\tilde{\sigma}_l)\right)dy\right)^2\\
&\qquad \leq  \left(\int_{\mathcal{Y}} |y| |\tilde{f}^{[k]}-\tilde{f}^{[l]}| dy \right)^2\\
&\qquad =\left(\int_{\mathcal{Y}}|y|\left(|\tilde{f}^{[k]}-\tilde{f}^{[l]}|\right)^{\frac{1}{2}}\left(|\tilde{f}^{[k]}-\tilde{f}^{[l]}|\right)^{\frac{1}{2}}dy \right)^2\\
&\qquad \overset{(a)}\leq  \int_{\mathcal{Y}}y^2  |\tilde{f}^{[k]}-\tilde{f}^{[l]}| dy \int_{\mathcal{Y}}|\tilde{f}^{[k]}-\tilde{f}^{[l]}|dy\\
&\qquad \overset{(b)}\leq \int_{\mathcal{Y}}y^2\left(\tilde{f}^{[k]}+\tilde{f}^{[l]}\right)dy \int_{\mathcal{Y}}|\tilde{f}^{[k]}-\tilde{f}^{[l]}|dy\\
&\qquad \overset{(c)}= \left(\tilde{\sigma}_k^2+\tilde{\mu}_k^2+\tilde{\sigma}_l^2+\tilde{\mu}_l^2\right)\int_{\mathcal{Y}}|\tilde{f}^{[k]}-\tilde{f}^{[l]}|dy\\
&\qquad =2\left(\tilde{\sigma}_k^2+\tilde{\mu}_k^2+\tilde{\sigma}_l^2+\tilde{\mu}_l^2\right)d_{TV}(\tilde{f}^{[k]},\tilde{f}^{[l]}),
\end{align*}
where $(a)$ follows by Cauchy-Schwarz inequality, $(b)$ using the identity $|a-b|\leq a+b$ provided that $a,b$ are positive, and $(c)$ using $E[x^2]\triangleq \int x^2f(x)dx$ and $E[x^2]\triangleq \mbox{Var}[x]+E[x]^2$. It follows that,
\begin{align}\label{upperbound}
&1-d_{TV}(\tilde{f}^{[k]},\tilde{f}^{[l]})\leq 1-\frac{\left(\int_{\mathcal{Y}}y\Big(\tilde{f}^{[k]}-\tilde{f}^{[l]}\Big)dy\right)^2}{2\Big(\tilde{\mu}_k^2+\tilde{\mu}_l^2+\tilde{\sigma}_k^2+\tilde{\sigma}_l^2\Big)}\nonumber\\
&=1-\frac{\Big(\tilde{\mu}_k-\tilde{\mu}_l\Big)^2}{2\Big(\tilde{\mu}_k^2+\tilde{\mu}_l^2+\tilde{\sigma}_k^2+\tilde{\sigma}_l^2\Big)}=\frac{1}{2}+\frac{2\tilde{\mu}_k\tilde{\mu}_l+\tilde{\sigma}_k^2+\tilde{\sigma}_l^2}{2(\tilde{\mu}_k^2+\tilde{\mu}_l^2+\tilde{\sigma}_k^2+\tilde{\sigma}_l^2)}.
\end{align}
Since the common area between each possible pair of model output PDFs is upper bounded by \eqref{upperbound}, then it follows that the common area between all pairwise model output PDFs is upper bounded by \eqref{upperboundgeneral}.
\end{proof}

The above result  provides an upper bound on the common area between the  model output PDFs, which can be evaluated in closed form. Next, we  deal with the TV distance constraint given by \eqref{TVconst}.

\begin{figure}[t]
        \centering
\includegraphics[width=.8\linewidth]{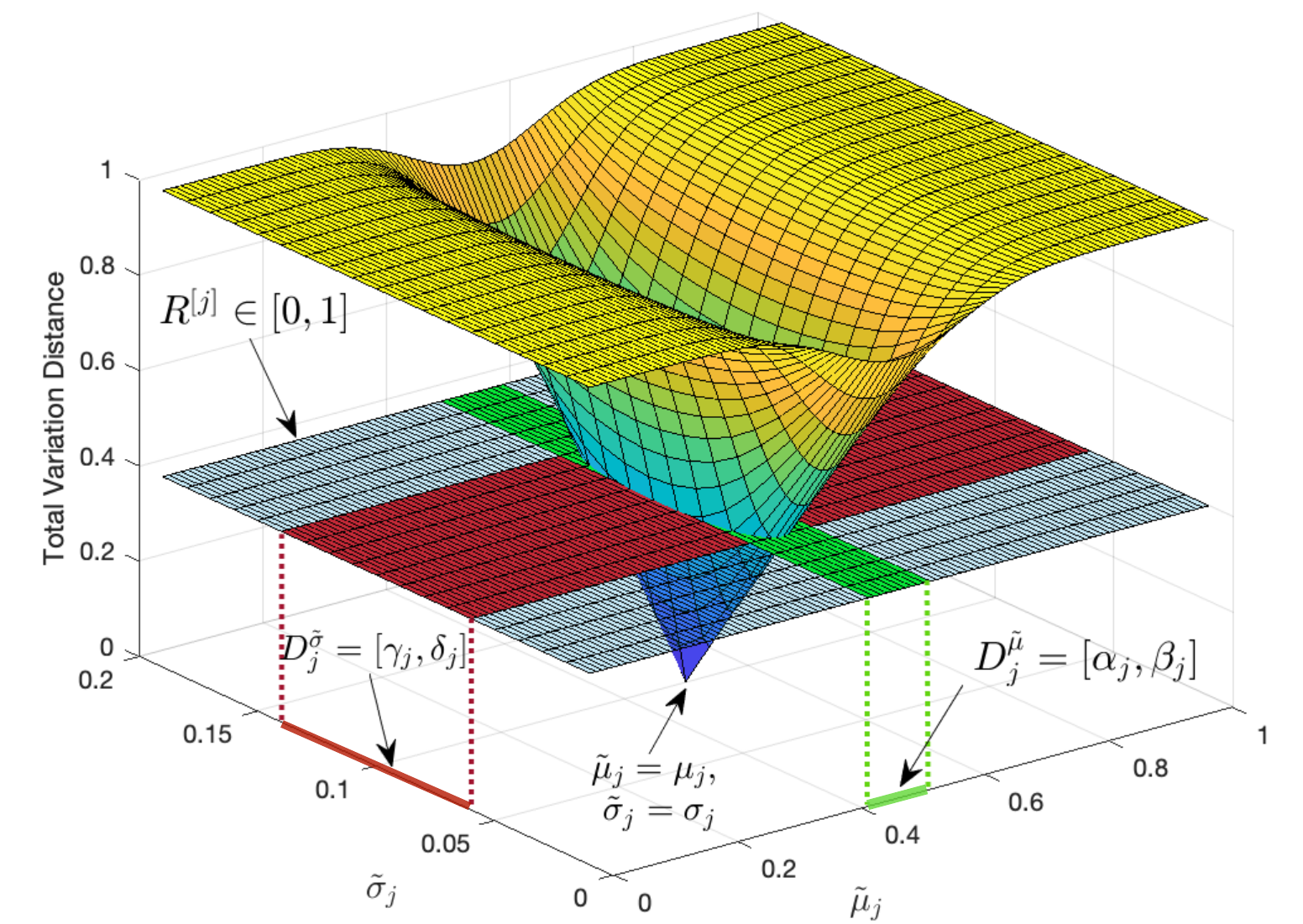}
\caption{TV distance constraint and the feasible ROI for a known nominal normal PDF $f^{[j]}(u,|\mu_j,\sigma_j)$ and  $R^{[j]}\in [0,1]$.}
\label{fig_upper_constraint}
\end{figure}

For each $j\in I$, given the nominal PDF $f^{[j]}(y|\mu_j,\sigma_j)$,  and the TV distance parameter $R^{[j]}\in [0,1]$, then one can evaluate the feasible regions of interest over which the pair  $(\tilde{\mu}_j,\tilde{\sigma}_j)$ achieves the maximum of the upper bound \eqref{upperboundgeneral}, and at the same time, guarantee that the TV distance constraint \eqref{TVconst} is not violated. In particular, the feasible regions of interest $ROI(j)$, $\forall j\in I$, can be defined as follows.
\begin{multline}\label{range.constraint}
ROI(j):\left\{(\tilde{\mu}_j,\tilde{\sigma}_j)\in \mathcal{X}\times\mathcal{X}:d_{TV}(\tilde{f}^{[j]},f^{[j]})\leq R^{[j]}\right\}\\ 
{\Longleftrightarrow} \tilde{\mu}_j\in D_j^{\tilde{\mu}}\triangleq[\alpha_j, \beta_j]\subseteq \mathcal{X}, \ \ \tilde{\sigma}_j\in D_j^{\tilde{\sigma}}\triangleq[\gamma_j, \delta_j]\subseteq \mathcal{X}.
\end{multline}
The fact that both $D_j^{\tilde{\mu}}\subseteq \mathcal{X}$ and $D_j^{\tilde{\sigma}}\subseteq \mathcal{X}$ follow from: (i) the state constraints of Problem \ref{problem2}, and (ii) variance upper bound results, such as,  if the random variable $X$ has a continuous unimodal PDF on $[a,b]$ with at least one mode in $[a,b]$, then $\sigma_X=(b-a)/3$ (see \cite{Seaman06}).
A graphical illustration of the TV distance constraint, as well as, the feasible regions of interest $D_j^{\tilde{\mu}}$ and $D_j^{\tilde{\sigma}}$, are shown in Figure \ref{fig_upper_constraint} for a known nominal normal PDF and with TV distance parameter  pre-selected within $[0,1]$. We note that, if the TV distance between the true and the nominal output PDFs is characterised through an equality constraint (i.e., see \cite{Tzortzis:2019}), then \eqref{TVconst} it is addressed through a feasible set of points. 

 By the above discussion, we have that both \eqref{max.problem} and \eqref{TVconst}, can be equivalently expressed as follows.
\begin{align*}
&\max_{\substack{\tilde{\mu}_j\in D_j^{\tilde{\mu}} \\ \tilde{\sigma}_j\in D_j^{\tilde{\sigma}}}}\sum_{i=0}^{n_f-1}\ \sum_{\mathclap{j=i+1}}^{n_f}\left(1-d_{TV}^{t_m}\left(\tilde{f}^{[i]}(y|\tilde{\mu}_i,\tilde{\sigma}_i),\tilde{f}^{[j]}(y|\tilde{\mu}_j,\tilde{\sigma}_j)\right)\right)\\
&=\sum_{i=0}^{n_f-1}\sum_{j=i+1}^{n_f}\left(1 -d_{TV}^{t_m}\left(\tilde{f}^{[i]}(y|\tilde{\mu}_i^*,\tilde{\sigma}^*_i),\tilde{f}^{[j]}(y|\tilde{\mu}_j^*,\tilde{\sigma}^*_j)\right)\right),\\
 &\quad  \mbox{for}\ (\tilde{\mu}_j^*,\tilde{\sigma}_j^*)\in \mathrlap{\arg}\max_{\tilde{\mu}_j\in D_j^{\tilde{\mu}},\ {\tilde{\sigma}_j\in D_j^{\tilde{\sigma}}}}\mbox{(Upper Bound  \eqref{upperboundgeneral})},\ \forall j\in I
\end{align*}
where $\arg\max(\cdot)$ it is the set of  $(\tilde{\mu}_j,\tilde{\sigma}_j)$ from the feasible regions of interest $D_j^{\tilde{\mu}}$, and $D_j^{\tilde{\sigma}}$, $\forall j\in I$, that achieve the maximum of the upper bound \eqref{upperboundgeneral}, and subsequently, are applied to evaluate the common area in terms of TV distance. This problem reformulation benefits since upper bound \eqref{upperboundgeneral} is expressed in terms of the distribution's first and second moments. This is an easier problem to solve than that of evaluating the distribution's density function, which for some distributions might be a tedious task. Moreover, since  the upper bound on the common area given by \eqref{upperboundgeneral} is concave and continuously differentiable with respect to $(\tilde{\mu},\tilde{\sigma})$, then the classic Langrangian method of multipliers can be applied. 

Toward this end, we define the Langrangian as follows
\begin{align}\label{langrangian}
\mathbb{L}(\tilde{\mu},\tilde{\sigma},\lambda^{\tilde{\mu}},s^{\tilde{\mu}},\lambda^{\tilde{\sigma}},s^{\tilde{\sigma}})&\triangleq  \sum_{i=0}^{n_f-1}\sum_{j=i+1}^{n_f}\left(\frac{1}{2}+\frac{2\tilde{\mu}_i\tilde{\mu}_j+\tilde{\sigma}_i^2+\tilde{\sigma}_j^2}{2(\tilde{\mu}_i^2+\tilde{\mu}_j^2+\tilde{\sigma}_i^2+\tilde{\sigma}_j^2)}\right)\\
&\quad +\sum_{j=0}^{n_f}\left(\lambda^{\tilde{\mu}}_j(\beta_j-\tilde{\mu}_j)+s^{\tilde{\mu}}_j(-\alpha_j+\tilde{\mu}_j)\right) +\sum_{j=0}^{n_f}\left(\lambda^{\tilde{\sigma}}_j(\delta_j-\tilde{\sigma}_j)+s^{\tilde{\sigma}}_j(-\gamma_j+\tilde{\sigma}_j)\right)\nonumber
\end{align}
where $\lambda^{\tilde{\mu}}_j,s^{\tilde{\mu}}_j$ and $\lambda^{\tilde{\sigma}}_j,s^{\tilde{\sigma}}_j$, $\forall j\in I$, are the Langrangian multipliers associated with the constraints given by \eqref{range.constraint}.
By the Karush-Kuhn-Tucker (KKT) theorem, the following conditions are necessary and sufficient for optimality
\begin{subequations}\label{comp}
\begin{align}
&\frac{\partial}{\partial \tilde{\mu}_j}\mathbb{L}|_{\tilde{\mu}=\tilde{\mu}^*,\lambda^{\tilde{\mu}}=\lambda^{\tilde{\mu},*},s^{\tilde{\mu}}=s^{\tilde{\mu},*},\lambda^{\tilde{\sigma}}=\lambda^{\tilde{\sigma},*}, s^{\tilde{\sigma}}= s^{\tilde{\sigma},*}}=0,  \label{comp01}\\
&\frac{\partial}{\partial \tilde{\sigma}_j}\mathbb{L}|_{\tilde{\mu}=\tilde{\mu}^*,\lambda^{\tilde{\mu}}=\lambda^{\tilde{\mu},*},s^{\tilde{\mu}}=s^{\tilde{\mu},*},\lambda^{\tilde{\sigma}}=\lambda^{\tilde{\sigma},*}, s^{\tilde{\sigma}}= s^{\tilde{\sigma},*}}=0,  \label{comp02}\\
&\tilde{\mu}_j^*-\beta_j\leq 0,\hspace{.5cm} \lambda^{\tilde{\mu},*}_j(\tilde{\mu}_j^*-\beta_j)=0, \hspace{.65cm} \lambda^{\tilde{\mu},*}_j\geq 0, \label{comp1}\\
&\tilde{\mu}_j^*-\alpha_j\geq 0,\hspace{.5cm} s^{\tilde{\mu},*}_j(-\tilde{\mu}_j^*+\alpha_j)= 0, \hspace{.45cm} s^{\tilde{\mu},*}_j\geq 0,\label{comp2}\\
&\tilde{\sigma}_j^*-\delta_j\leq 0,\hspace{.5cm} \lambda^{\tilde{\sigma},*}_j(\tilde{\sigma}_j^*-\delta_j)=0,  \hspace{.75cm} \lambda^{\tilde{\sigma},*}_j\geq 0, \label{comp6}\\
&\tilde{\sigma}_j^*-\gamma_j\geq 0,\hspace{.5cm} s^{\tilde{\sigma},*}_j(-\tilde{\sigma}_j^*+\gamma_j)= 0,  \hspace{.45cm} \ s^{\tilde{\sigma},*}_j \geq 0,\label{comp8}
\end{align}
\end{subequations}
for all $j\in I$. By \eqref{comp01} and \eqref{comp02}, differentiating \eqref{langrangian} with respect to $\tilde{\mu}$ and  $\tilde{\sigma}$, respectively, the following equations are obtained
\begin{subequations}
\begin{align}
&\frac{\partial}{\partial \tilde{\mu}_j}\mathbb{L}|_{\tilde{\mu}=\tilde{\mu}^*,\lambda^{\tilde{\mu}}=\lambda^{\tilde{\mu},*},s^{\tilde{\mu}}=s^{\tilde{\mu},*},\lambda^{\tilde{\sigma}}=\lambda^{\tilde{\sigma},*}, s^{\tilde{\sigma}}= s^{\tilde{\sigma},*}}\nonumber\\
&\qquad=\sum_{\substack{i\in I\\ i\neq j}}\frac{(\tilde{\mu}^*_i-\tilde{\mu}^*_j)\Big((\tilde{\mu}^*_i)^2+\tilde{\mu}^*_i\tilde{\mu}^*_j+(\tilde{\sigma}^*_i)^2+(\tilde{\sigma}^*_j)^2\Big)}{\Big((\tilde{\mu}^*_j)^2+(\tilde{\mu}^*_i)^2+(\tilde{\sigma}^*_i)^2+(\tilde{\sigma}^*_j)^2\Big)^2}+s^{\tilde{\mu},*}_j-\lambda^{\tilde{\mu},*}_j=0,\\
&\frac{\partial}{\partial \tilde{\sigma}_j}\mathbb{L}|_{\tilde{\mu}=\tilde{\mu}^*,\lambda^{\tilde{\mu}}=\lambda^{\tilde{\mu},*},s^{\tilde{\mu}}=s^{\tilde{\mu},*},\lambda^{\tilde{\sigma}}=\lambda^{\tilde{\sigma},*}, s^{\tilde{\sigma}}= s^{\tilde{\sigma},*}}\nonumber\\
&\qquad=\sum_{\substack{i\in I\\ i\neq j}}\frac{\tilde{\sigma}^*_j(\tilde{\mu}^*_j-\tilde{\mu}^*_i)^2}{\Big((\tilde{\mu}^*_j)^2+(\tilde{\mu}^*_i)^2+(\tilde{\sigma}^*_j)^2+(\tilde{\sigma}^*_i)^2\Big)^2}+s^{\tilde{\sigma},*}_j-\lambda^{\tilde{\sigma},*}_j=0.
\end{align}
\end{subequations}
By  the complementary conditions \eqref{comp1}-\eqref{comp8}, for each $j\in I\triangleq \{0,1,\dots,n_f\}$, the optimal solution $\tilde{\mu}^*_j$ ($\tilde{\sigma}^*_j$) may be  either at the boundary, in which case, $\tilde{\mu}^*_j-\beta_j=0$ or $-\tilde{\mu}^*_j+\alpha_j=0$ ($\tilde{\sigma}^*_j-\delta_j=0$ or $-\tilde{\sigma}^*_j+\gamma_j=0$), or at the interior, in which case,  $\lambda^{\tilde{\mu},*}_j,s^{\tilde{\mu},*}_j=0$ ($\lambda^{\tilde{\sigma},*}_j,s^{\tilde{\sigma},*}_j=0$).
By the concavity of the objective function and by the linearity of the constraints an optimal solution always exists, that is, given $(\tilde{\mu}^*,\tilde{\sigma}^*)$ with $s^{\bullet,*}$ and $\lambda^{\bullet,*}$ satisfying the KKT conditions, then the pair $(\tilde{\mu}^*,\tilde{\sigma}^*)$ maximizes the objective function subject to the constraints. 
\begin{remark}
The proposed distributionally robust AFD approach is not limited to normal approximations. Indeed, as the next example shows, the results of this work can be readily transferred to alternative distributions other than normal, for which the first and second moments exist. This fact is illustrated in more detail for two well-known distributions, namely, (a) the beta distribution, and (b) the gamma distribution.
\end{remark}
\begin{example}\label{toyexample}
Consider the case in which the resulting probability histograms  are approximated by a gamma distribution with shape parameter $\tilde{a}$ and rate parameter $\tilde{b}$. Given that the nominal PDFs are obtained by a gamma approximation then it follows that the true output PDFs, $\tilde{f}^{[j]}$, $\forall j\in I$, are given by
\begin{equation*}
\tilde{f}^{[j]}(y|\tilde{a}_j,\tilde{b}_j)=\frac{\tilde{b}_j(\tilde{b}_jy)^{\tilde{a}_j-1}e^{-\tilde{b}_jy}}{\Gamma(\tilde{a}_j)},\quad y\in (0,\infty),\quad \tilde{a}_j,\tilde{b}_j>0
\end{equation*}
where $\Gamma(\tilde{a}_j)$ is the gamma function. The mean and the variance of a gamma distribution  are given by $\tilde{a}_j/\tilde{b}_j$ and $\tilde{a}_j/\tilde{b}^2_j$, respectively. By applying the proposed approach to obtain the optimal pair $(\tilde{\mu}^*_j,\tilde{\sigma}^*_j)$, $j\in I$, then the shape and rate parameters that produce the desired gamma distribution are given by:
\begin{equation*}
\tilde{a}_j=\left(\frac{\tilde{\mu}_j^*}{\tilde{\sigma}_j^*}\right)^2,\quad \mbox{and}\quad \tilde{b}_j=\frac{\tilde{\mu}_j^*}{\left(\tilde{\sigma}_j^*\right)^2}.
\end{equation*}
Another useful distribution for histogram approximation over finite intervals is the beta distribution with shape parameters $(\tilde{a},\tilde{b})$. The mean and the variance of the beta distribution are given by $\tilde{a}_j/(\tilde{a}_+\tilde{b}_j)$ and $\tilde{a}_j\tilde{b}_j/(\tilde{a}_j+\tilde{b}_j)^2(\tilde{a}_j+\tilde{b}_j+1)$, respectively. Following the proposed approach for finding the optimal pair  $(\tilde{\mu}^*,\tilde{\sigma}^*)$, then the shape parameters that produce the desired beta distribution are given by:
\begin{equation*}
\tilde{a}_j=\frac{\tilde{\mu}_j^*\left(\tilde{\mu}_j^*-\left(\tilde{\mu}_j^*\right)^2-\left(\tilde{\sigma}_j^*\right)^2\right)}{\left(\tilde{\sigma}_j^*\right)^2},\quad \mbox{and}\quad \tilde{b}_j=\frac{\tilde{a}_j\left(1-\tilde{\mu}_j^*\right)}{\tilde{\mu}_j^*}.
\end{equation*}
\end{example}

\subsection{Robust AFD algorithm and implementation of the AFD scheme in practice}\label{sec.algorithm}
In this section, we first give Robust AFD  Algorithm \ref{algorithm}  to solve  Problem \ref{problem2}. Then,  AFD procedure \ref{procedure1} is provided and its implementation in practice is discussed.
\begin{algorithm}[htbp]
\caption{Robust AFD Algorithm}\label{algorithm}
Input data: 1) $T$: time horizon, 2) $m^{[j]}$: faultless and faulty models,  3) $x^{[j]}(0)$: initial states, 4) nominal distribution information of  $\theta^{[j]}$, 5)  $R^{[j]}\in [0,1]$: TV distance parameter, 6) input and state constraints. 

At each  iteration of  outer optimization in \eqref{opt.cost.problem2}, do:
\begin{enumerate}
\item[1:] For each model $j\in I$:
\begin{enumerate}
\item using the initial states and the nominal  distribution information of  $\theta^{[j]}$,  perform Monte Carlo simulations of  \eqref{eq1}, to obtain the probability histograms under parametric uncertainties,
\item approximate  the resulting probability histograms to construct the nominal output PDFs, $f^{[j]}$,
\item using  $R^{[j]}$ and the nominal PDFs, $f^{[j]}$,  evaluate the bounded regions of interest, $D^{\tilde{\mu}}_j$ and $D^{\tilde{\sigma}}_j$, as described in \eqref{range.constraint}.
\end{enumerate}
\item[2:]  At each  iteration of  inner optimization in \eqref{opt.cost.problem2}, do:
\begin{enumerate}
\item For each model  $j\in I$, using $\tilde{\mu}_j\in D^{\tilde{\mu}}_j$ and $\tilde{\sigma}_j\in D^{\tilde{\sigma}}_j$,   compute the upper bound given by the right-hand side of \eqref{upperboundgeneral}.
\item Solve the inner optimization in \eqref{opt.cost.problem2}, by applying the KKT conditions \eqref{comp}, to obtain $\tilde{\mu}_j^*\in D^{\tilde{\mu}}_j$ and $\tilde{\sigma}_j^*\in D^{\tilde{\sigma}}_j$, for all $j\in I$. 
\end{enumerate}
\item[3:] Compute the common area between the true output PDFs $\tilde{f}^{[j]}$, for all $j\in I$, as given by \eqref{opt.cost.problem2}.
\item[4:] Solve the outer optimization in \eqref{opt.cost.problem2} to obtain a distributionally robust input $u^*\in \mathcal{U}$.
\end{enumerate}
\end{algorithm}

\begin{figure*}[t]
\centering
\includegraphics[width=\linewidth]{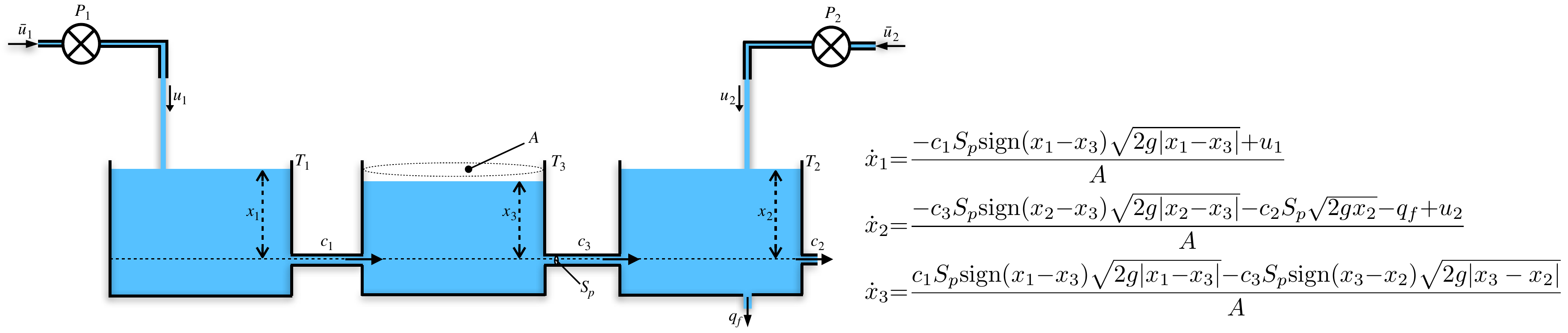}
\caption{The three-tank system}\label{fig_new}
\end{figure*}


%
 To exemplify the implementation of the proposed AFD scheme, next we provide the procedure to be followed.
\begin{procedure}\label{procedure1} Let $t_{m}=t_{m_1},t_{m_2},\hdots\in [0,T]$ denote the measurements time in which observations are to be taken. Also, let $k=0$ and $t_{m_0}=0$. Then:\begin{description}[\setlabelwidth{Step 1:}]
\item[Step 1:] At time $t=t_{m_k}$, call Algorithm \ref{algorithm} (with the updated state values) to design an optimal input signal $u^*(t)\in {\mathcal U}$, $t\in[t_{m_k},t_{m_{k+1}}]$, which enables the separation of multiple models at  time $t=t_{m_{k+1}}$.
\item[Step 2:] Using the designed optimal input signal $u^*(t)$ of Step 1, simulate the multiple models from time $t=t_{m_{k}}$ up to time $t=t_{m_{k+1}}$, and save the state values $x^{[j]}(t_{m_{k+1}})=x^{[j]}_{t_{m_{k+1}}}$.
\item[Step 3:] Let $k=k+1$, and repeat steps 1-3. 
\end{description}\end{procedure}
Through the proposed offline AFD procedure a distributionally robust input signal $u^*(t)$, $t\in [0,T]$, is obtained, which when applied to the true system enables the detection and isolation of faults by comparing the true system measurements with the faultless and faulty model predictions. Based on the proposed AFD procedure, the overall function of the designed input signal  $u(t)\in {\mathcal U}$ can be thought as a  piecewise function of the following form
\begin{equation*}
u(t)=\begin{cases}
              \bar{u}(0),\quad\quad \mbox{if}\ t<t_{m_1}\\
               \bar{u}(t_{m_k}),\quad \mbox{if}\ t_{m_k}\leq t<t_{m_{k+1}},\ k{=}1,\dots,|m|{-}1\\
              \bar {u}(t_{m_{|m|}}),\  \mbox{if}\ t_{m_{|m|}}\leq t<T
            \end{cases}
\end{equation*}
with $|m|$ denoting the number of different time instances.


Next, we consider a three-tank benchmark problem for AFD, drawn from \cite{Mesbah:2014}, and modified to illustrate the effect of highly uncertain and ambiguous distribution information. Further information  regarding the three-tank system can also be found in \cite{Zhang:2002,Seliger:1999}.

\section{Application: Robust AFD for the three-tank system}\label{sec.example}



\begin{table}[b]
\begin{minipage}[b]{.485\linewidth}
   \centering
      \caption{Nominal   parameters}
   \begin{tabular}{ |c| }
     \hline
     $c_1\sim \mathcal{N}(1,0.0025)$ \\ \hline
      $c_2\sim \mathcal{N}(0.8,0.0025)$ \\ \hline
      $c_3\sim \mathcal{N}(1,0.0025)$ \\ \hline
      $r\sim\mathcal{N}(0.002,10^{-6})$\\ \hline
      $\alpha\sim \mathcal{N}(0.6,4\times 10^{-4})$\\ [1ex]
     \hline
   \end{tabular}
   \label{table}
\end{minipage}\
\begin{minipage}[b]{.485\linewidth}
   \centering
      \caption{True  parameters}
   \begin{tabular}{ | c | }
     \hline
     $c_1\sim \mathcal{N}(1,0.01)$ \\ \hline
      $c_2\sim \mathcal{N}(1,0.01)$ \\ \hline
      $c_3\sim \mathcal{N}(1,0.01)$ \\ \hline
      $r\sim\mathcal{N}(0.02,10^{-6})$\\ \hline
      $\alpha\sim \mathcal{N}(0.6,4\times 10^{-2})$\\ [1ex]
     \hline
   \end{tabular}
   \label{table1}
\end{minipage}
\end{table}

A three-tank system consisting of three identical cylinders $T_1$, $T_2$, and $T_3$, with cross section area $A=0.0154m^2$, is depicted in Figure \ref{fig_new}. The tanks are interconnected by connection pipes with cross section area $S_p=5\times10^{-5}m^2$, and the liquid entering tanks $T_1$ and $T_2$ via the pumps $P_1$ and $P_2$, respectively. The liquid flow rate from pump $P_i$, $i=1,2$, is denoted by $u_i$, and is constrained to take values in  $0\leq u_i\leq 10^{-4}\ m^3/s$, $i=1,2$.
The three-tank system is modeled using mass balance equations and Toricelli's law, as shown in Figure \ref{fig_new}, where $x_i$, $i=1,2,3$, denote the  liquid levels in tanks $T_1$, $T_2$, and $T_3$, respectively, and are constrained to take values in $0\leq x_i\leq 0.75\ m$, $i=1,2,3$. In addition, $g=9.81m/s^2$ denotes the gravity acceleration,  and $q_f$ denotes the outflow rate from tank $T_2$ due to leakage ($q_f=0$ in the faultless case). Moreover, $c_i$, $i=1,2,3,$  denotes the nondimensional outflow  coefficients which are considered to be uncertain. We assume that only the liquid level in tank $T_3$ is available for measurement, i.e., $y^{[i]}(t)=x_3^{[i]}(t),\forall i \in I$, and is corrupted by white noise, which has a normal distribution with zero mean and variance $0.025$. 

Three operation scenarios are considered, the faultless scenario  and two faulty scenarios:
\begin{description}[\setlabelwidth{Fault A}]
\item[Fault A.]  A multiplicative actuator fault in pump $P_1$ defined by  $u_1=\bar{u}_1+(\alpha-1)\bar{u}_1$, where $\alpha$ is an uncertain parameter characterizing the magnitude of the fault, and $\bar{u}_1$ is the supply flow rate in the faultless case.
\item[Fault B.] A leakage in tank $T_2$ with a circular leak of uncertain radius $r$, and outflow rate $q_f=c_2\pi r^2\sqrt{2gx_2}$.
\end{description}
To illustrate the effect of ambiguous distribution information, we choose those uncertain model parameters which are of interest to describe the system behaviour under the different fault scenarios. 
In particular, the nominal and true distribution information of uncertain model parameters is summarized in Tables \ref{table} and  \ref{table1}, respectively.  In all simulations, the time horizon is set equal to $T=3000s$, and the  initial conditions $x_i(0)$, $i=1,2,3$, are distributed uniformly on the interval $[0,0.15]$. By numerical integration  of the three-tank system (given in Figure \ref{fig_new}) in Matlab,  the histogram of the output is obtained for the three operation scenarios.
In particular, histograms are constructed based on $10000$ Monte Carlo simulations of the three-tank system under parametric uncertainties. The corresponding PDFs are obtained by normal approximation using Matlab's function \emph{fitdist}.  In addition, for the design of the optimal input signal $u^*\in \mathcal{U}$,  under the faultless and faulty operation scenarios, Matlab's global optimization toolbox is employed, while the KKT conditions are applied using the \emph{fsolve} solver.
In what follows, we compare AFD input design results under two possible cases: 
\begin{enumerate}
\item[(C1)] solution of the AFD optimization problem using the nominal distribution information (as listed in Table \ref{table}), and with TV distance parameter set equal to $R^{[j]}=0$ for all $j=1,2,3$. Typically, here we assume that the nominal distribution information provided for uncertain model parameters is correct. We refer to the input signal designed under this  case as the \emph{nominal input signal}, and 
\item[(C2)] solution of the AFD optimization problem using the nominal distribution information (as listed in Table \ref{table}), and with TV distance parameter set equal to $R^{[j]}=1$ for all $j=1,2,3$. Typically, here we assume  that the nominal distribution information provided for uncertain model parameters is poor and not accurate enough. We refer to the input signal designed under this case  as the \emph{distributionally robust input signal}.
\end{enumerate}
For cases (C1) and (C2), measurements are taken every $100s$. The  computational time  required to solve the optimization problem and obtain the nominal input signal under (C1), and the distributionally robust input signal under (C2), is approximately 35  and 50 minutes, respectively (using an iMac, Quad-Core i7, 4.2GHz, 32GB).

Figure \ref{Hill3}\subref{fig3.1}-\subref{fig3.2} depicts the nominal  and the distributionally robust input signals designed under cases (C1) and (C2), respectively. For illustrating the effectiveness of the proposed AFD method, next we apply  to the true model (described by the true distribution information, as listed in Table \ref{table1}) both the nominal input signal and the distributionally robust input signal.
Figure \ref{Hill3}\subref{fig3.3}-\subref{fig3.4} depicts the  common area, in terms of TV distance metric, evaluated at each measurement time under both designed input signals. 
Figure \ref{Hill3}\subref{fig3.4} confirms that the distributionally robust input signal restricts the influences of ambiguous distribution information, while Figure \ref{Hill3}\subref{fig3.3} confirms that the nominal input signal does not.  
In addition, Figure \ref{Hill3}\subref{fig2.3}-\subref{fig2.4} depicts the estimated output PDFs of the liquid level in tank $T_3$ at a single time instant, $t_m=3000s$. Comparing Figure \ref{Hill3}\subref{fig2.3} with Figure \ref{Hill3}\subref{fig2.4}, it can be seen that under the  nominal input signal the estimated  output PDFs of the true model share a common area of size $a\approx 0.25$, while under the distributionally robust input signal the corresponding PDFs share a common area of size $a\approx 0$.
This is due to the fact that the nominal input signal is designed only for the case in which the true and the nominal output PDFs are equal, and hence, effective fault diagnosability of the real system is not ensured. In contrast, to  nominal input signal, the distributionally robust input signal is designed by considering the worst-case PDFs over all possible PDFs within the TV distance ambiguity set, and hence, effective fault diagnosability of the real system is ensured even in  cases of highly uncertain scenarios.
\begin{figure}[htbp]
\centering
\subfloat[][Nominal input signal]{
\label{fig3.1} 
\includegraphics[ width=0.48\linewidth]{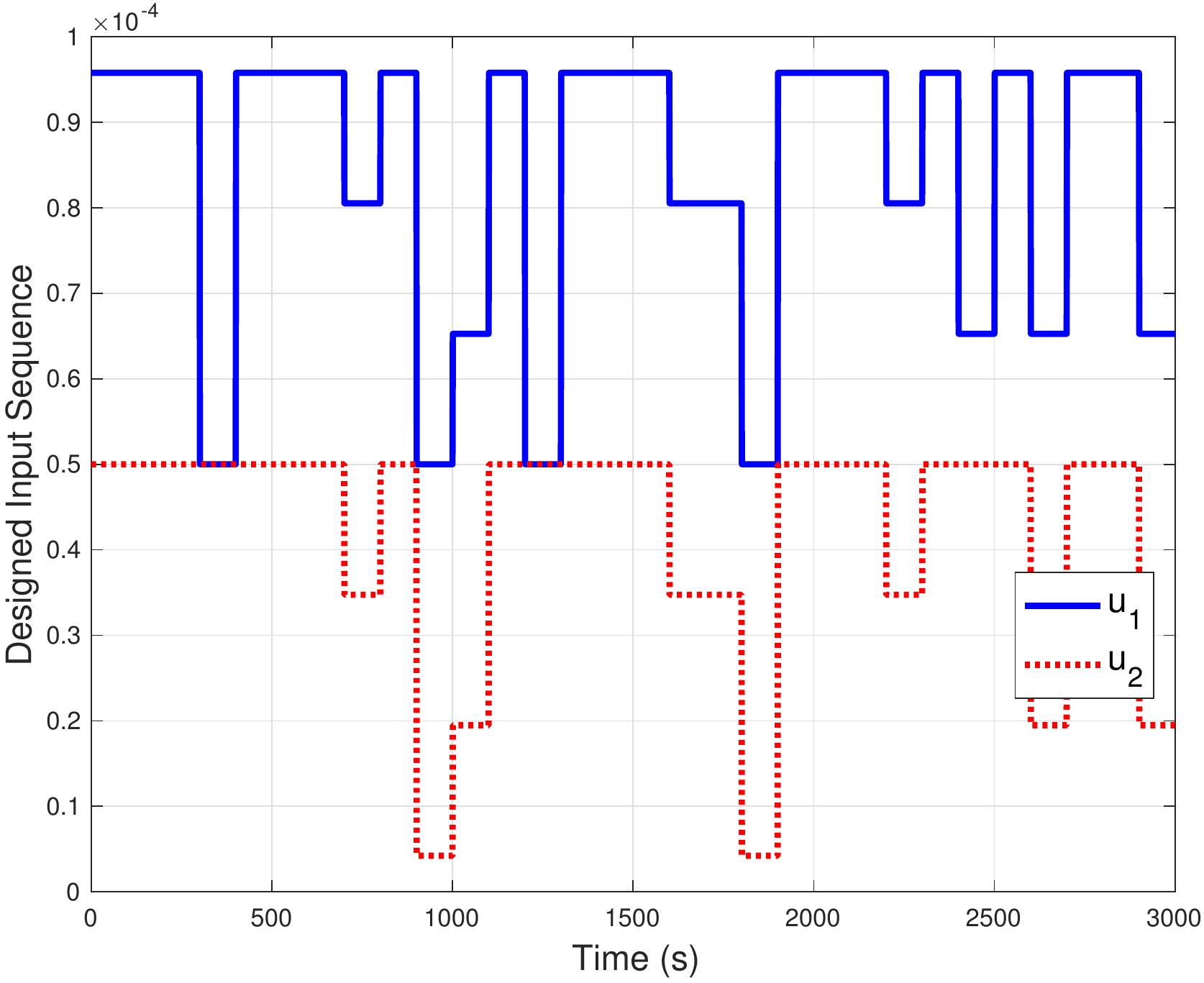}}\hspace{.02cm}
\subfloat[][Distributionally robust input signal]{
\label{fig3.2} 
\includegraphics[width=0.48\linewidth]{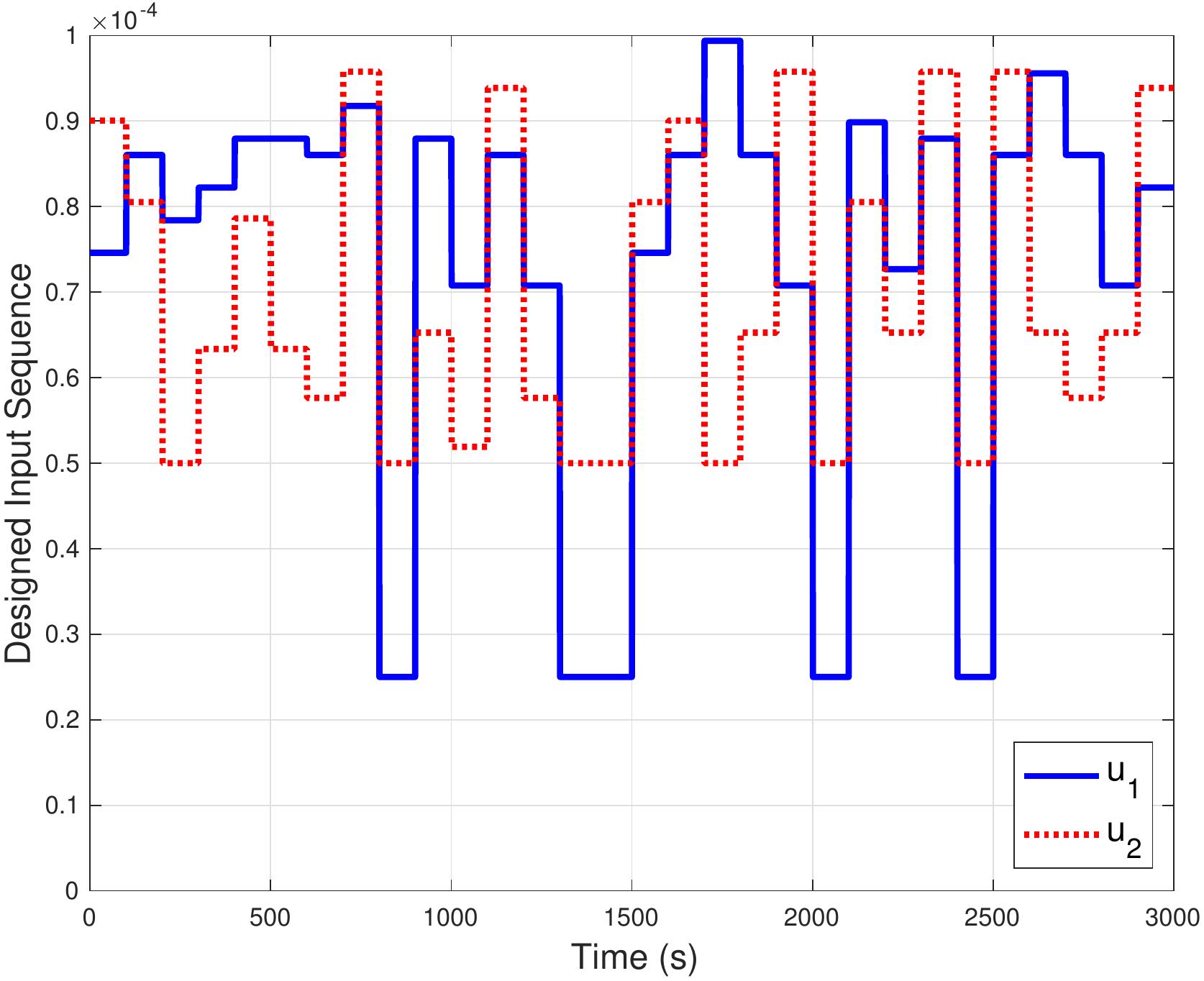}}\\
\subfloat[][PDFs under the nominal input signal]{
\label{fig2.3} 
\includegraphics[ width=0.48\linewidth]{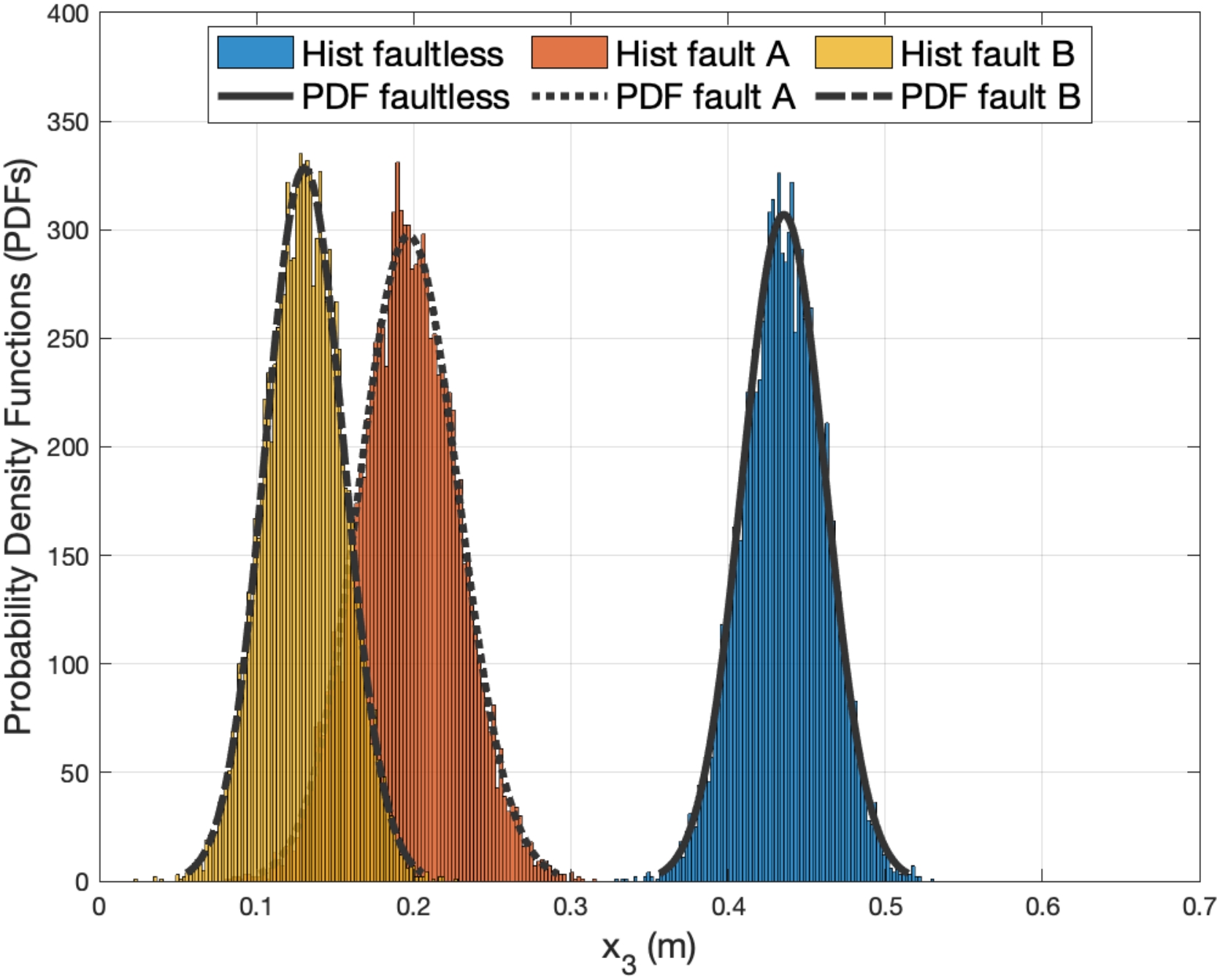}}\hspace{.02cm}
\subfloat[ttt][PDFs under the distributionally robust input signal]{
\label{fig2.4} 
\includegraphics[width=0.48\linewidth]{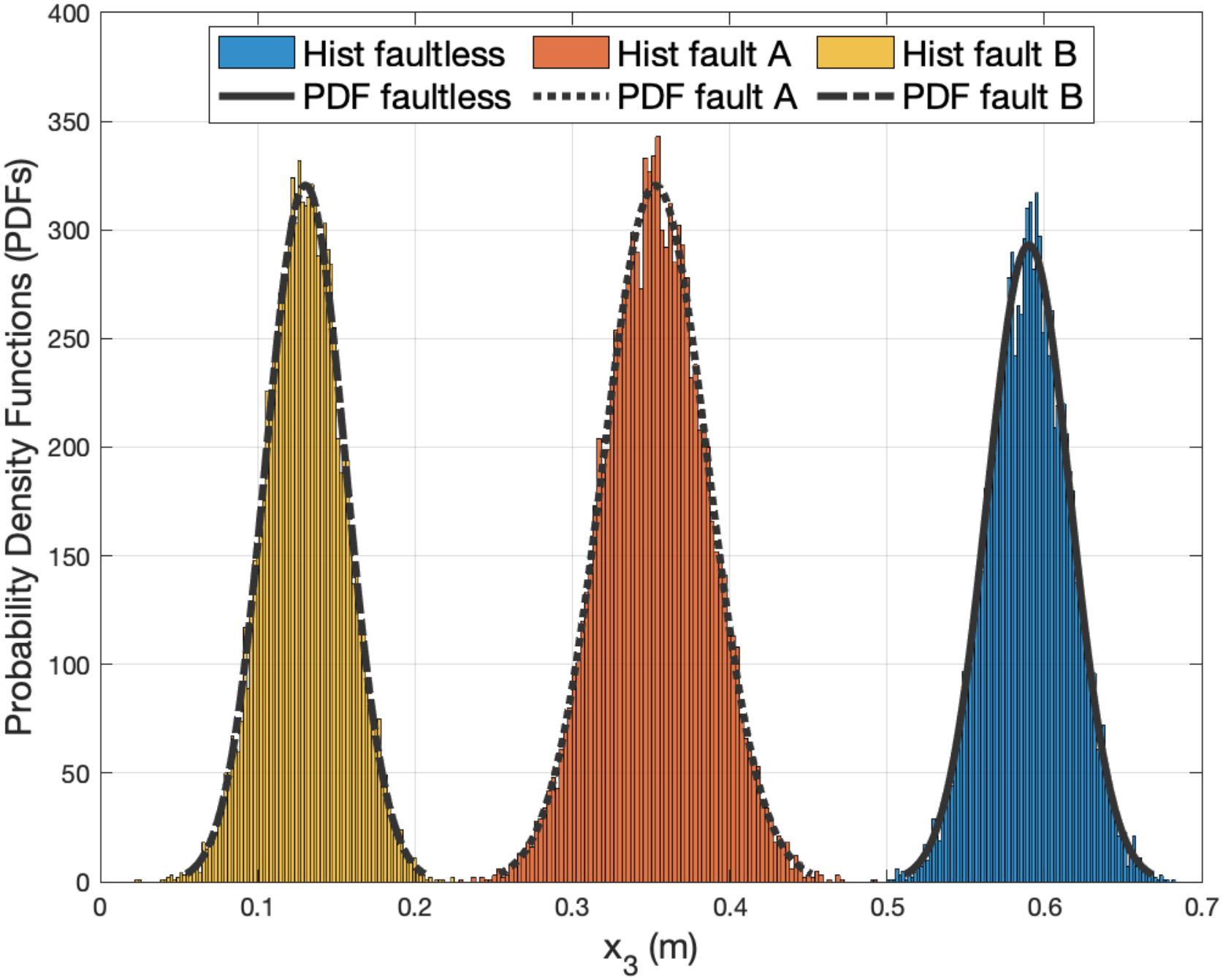}}\hfill
\subfloat[][Area under the nominal input signal]{
\label{fig3.3} 
\includegraphics[ width=0.48\linewidth]{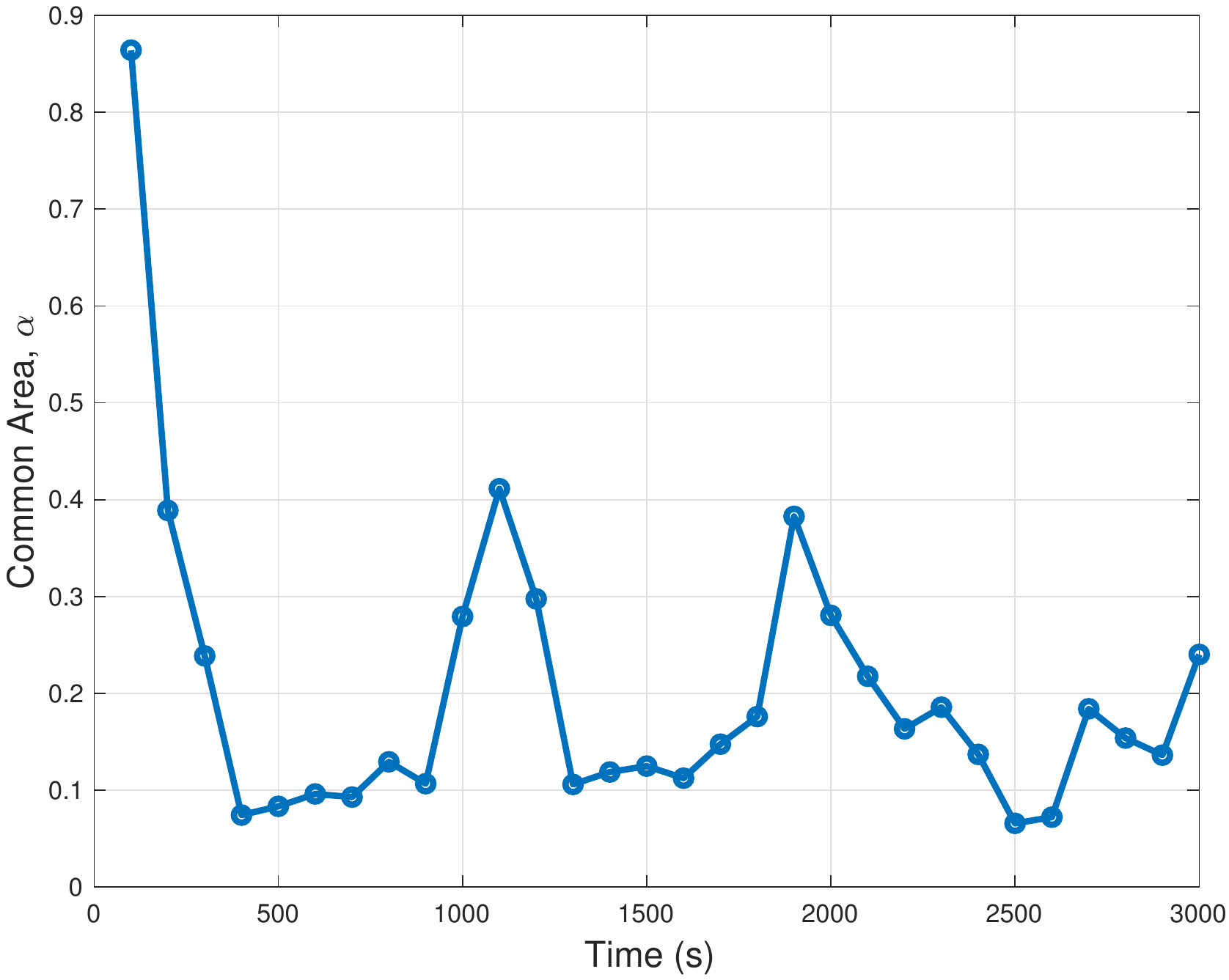}}\hspace{.02cm}
\subfloat[][Area under the distributionally robust input signal]{
\label{fig3.4} 
\includegraphics[width=0.48\linewidth]{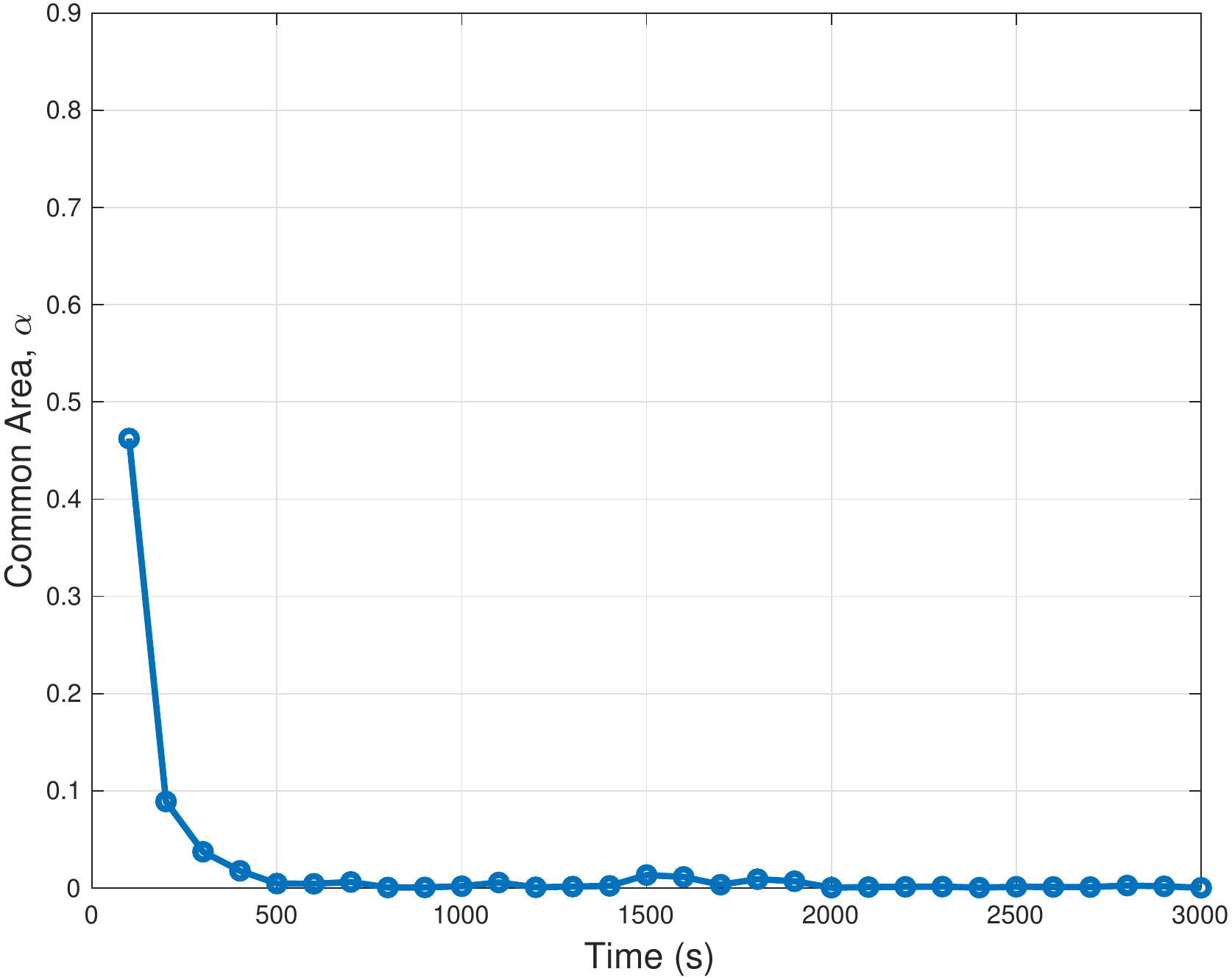}}
\caption[]{Top plots depict the input design signals. Middle plots depict the true output PDFs at measurement time  $t_m=3000s$. Bottom plots depict the resulting  common area at each measurement time.}
\label{Hill3}
\end{figure}

\section{Conclusion}\label{sec.conclusion}
In this paper, we have proposed a distributionally robust,  AFD approach for a class nonlinear systems with respect to ambiguous distribution information of uncertain model parameters. With respect to existing literature, TV distance metric is introduced as a measure for the evaluation of the common area between distributions corresponding to multiple fault scenarios, and as an information constraint to model ambiguity. Results indicate the capability of the proposed approach to restrict the influences of uncertainty, and its effectiveness on the detection and isolation of faults. 

Future research directions include the extension of the results to the case where the distributionally robust separating input signal is applied as an auxiliary to the nominal/stabilizing control input signal, and to study possible side effects to the system behaviour. In addition, to hedge against the impact of inaccuracy in PDF estimation by working directly on the resulting probability histograms, utilizing the alternative definition of the TV distance metric as given by \eqref{tvdfn}.

\flushend

\bibliographystyle{IEEEtran}
\bibliography{autosam}

\begin{thebibliography}{10}
\providecommand{\url}[1]{#1}
\csname url@samestyle\endcsname
\providecommand{\newblock}{\relax}
\providecommand{\bibinfo}[2]{#2}
\providecommand{\BIBentrySTDinterwordspacing}{\spaceskip=0pt\relax}
\providecommand{\BIBentryALTinterwordstretchfactor}{4}
\providecommand{\BIBentryALTinterwordspacing}{\spaceskip=\fontdimen2\font plus
\BIBentryALTinterwordstretchfactor\fontdimen3\font minus
  \fontdimen4\font\relax}
\providecommand{\BIBforeignlanguage}[2]{{%
\expandafter\ifx\csname l@#1\endcsname\relax
\typeout{** WARNING: IEEEtran.bst: No hyphenation pattern has been}%
\typeout{** loaded for the language `#1'. Using the pattern for}%
\typeout{** the default language instead.}%
\else
\language=\csname l@#1\endcsname
\fi
#2}}
\providecommand{\BIBdecl}{\relax}
\BIBdecl

\bibitem{Tzortzis:2019}
I.~Tzortzis and M.~Polycarpou, ``Distributionally robust active fault
  diagnosis,'' in \emph{Proc. of the European Control Conference}, 2019, pp.
  3886--3891.

\bibitem{Nikoukhah:2006}
R.~Nikoukhah and S.~Campbell, ``Auxiliary signal design for active failure
  detection in uncertain linear systems with a priori information,''
  \emph{Automatica}, vol.~42, no.~2, pp. 219--228, 2006.

\bibitem{Andjelkovic:2008}
I.~Andjelkovic, K.~Sweetingham, and S.~Campbell, ``Active fault detection in
  nonlinear systems using auxiliary signals,'' in \emph{Proc. of the American
  Control Conference}, 2008, pp. 2142--2147.

\bibitem{Scott:2013}
J.~Scott, G.~Marseglia, L.~Magni, R.~Braatz, and D.~Raimondo, ``A hybrid
  stochastic-deterministic input design method for active fault diagnosis,'' in
  \emph{52nd IEEE Conference on Decision and Control}, 2013, pp. 5656--5661.

\bibitem{Raimondo:2013}
D.~Raimondo, R.~Braatz, and J.~Scott, ``Active fault diagnosis using moving
  horizon input design,'' in \emph{Proc. of the European Control Conference},
  2013, pp. 3131--3136.

\bibitem{Paulson:2014}
J.~Paulson, D.~Raimondo, R.~Findeisen, R.~Braatz, and S.~Streif, ``Guaranteed
  active fault diagnosis for uncertain nonlinear systems,'' in \emph{Proc. of
  the European Control Conference}, 2014.

\bibitem{MARSEGLIA2017}
G.~Marseglia and D.~Raimondo, ``Active fault diagnosis: A multi-parametric
  approach,'' \emph{Automatica}, vol.~79, pp. 223--230, 2017.

\bibitem{Yang21}
S.~{Yang}, F.~{Xu}, X.~{Wang}, and B.~{Liang}, ``A novel online active fault
  diagnosis method based on invariant sets,'' \emph{IEEE Control Systems
  Letters}, vol.~5, no.~2, pp. 457--462, 2021.

\bibitem{XU2021}
F.~Xu, ``Observer-based asymptotic active fault diagnosis: A two-layer
  optimization framework,'' \emph{Automatica}, vol. 128, p. 109558, 2021.

\bibitem{Yao20}
L.~Yao and Y.~Wu, ``Robust fault diagnosis and fault-tolerant control for
  uncertain multiagent systems,'' \emph{International Journal of Robust and
  Nonlinear Control}, vol.~30, no.~18, pp. 8192--8205, 2020.

\bibitem{Nikoukhah:2000}
R.~Nikoukhah, S.~Campbell, and F.~Delebecque, ``Detection signal design for
  failure detection: a robust approach,'' \emph{International Journal of
  Adaptive Control and Signal Processing}, vol.~14, pp. 701--724, 2000.

\bibitem{Ashari:2012}
A.~Ashari, R.~Nikoukhah, and S.~Campbell, ``Active robust fault detection in
  closed-loop systems: Quadratic optimization approach,'' \emph{IEEE Trans.
  Autom. Control}, vol.~57, no.~10, pp. 2532--2544, 2012.

\bibitem{Chen:2012}
J.~Chen and R.~Patton, \emph{Robust Model-Based Fault Diagnosis for Dynamic
  Systems}.\hskip 1em plus 0.5em minus 0.4em\relax Springer Publishing Company,
  Incorporated, 2012.

\bibitem{Streif:2013}
S.~Streif, D.~Hast, R.~Braatz, and R.~Findeisen, ``Certifying robustness of
  separating inputs and outputs in active fault diagnosis for uncertain
  nonlinear systems,'' in \emph{Proc. of International Symposium on Dynamics
  and Control of Process Systems}, 2014, pp. 837--842.

\bibitem{Zhang:2002}
X.~Zhang, M.~Polycarpou, and T.~Parisini, ``A robust detection and isolation
  scheme for abrupt and incipient faults in nonlinear systems,'' \emph{IEEE
  Trans. Autom. Control}, vol.~47, no.~4, pp. 576--593, 2002.

\bibitem{Kerestecioglu:1993}
F.~Kerestecioglu, \emph{Change Detection and Input Design in Dynamical
  Systems}.\hskip 1em plus 0.5em minus 0.4em\relax New York: John Wiley, 1993.

\bibitem{Blackmore:2008}
L.~Blackmore, S.~Rajamanohran, and B.~Williams, ``Active estimation for jump
  {M}arkov linear systems,'' \emph{IEEE Trans. Autom. Control}, vol.~53,
  no.~10, pp. 2223--2236, 2008.

\bibitem{Mesbah:2014}
A.~Mesbah, S.~Streif, R.~Findeisen, and R.~Braatz, ``Active fault diagnosis for
  nonlinear systems with probabilistic uncertainties,'' in \emph{Proceedings of
  the 19th IFAC World Congress}, 2014, pp. 7079--7084.

\bibitem{Mesbah:2014b}
S.~Streif, F.~Petzke, A.~Mesbah, R.~Findeisen, and R.~Braatz, ``Optimal
  experimental design for probabilistic model discrimination using polynomial
  chaos,'' in \emph{Proceedings of the 19th IFAC World Congress}, 2014, pp.
  4103--4109.

\bibitem{Poulsen:2008}
N.~Poulsen and H.~Niemann, ``Active fault diagnosis based on stochastic
  tests,'' \emph{Int. J. Appl. Math. Comput. Sci.}, vol.~18, no.~4, pp.
  487--496, 2008.

\bibitem{CHENG2021107353}
P.~Cheng, M.~Chen, V.~Stojanovic, and S.~He, ``Asynchronous fault detection
  filtering for piecewise homogenous {M}arkov jump linear systems via a dual
  hidden {M}arkov model,'' \emph{Mechanical Systems and Signal Processing},
  vol. 151, p. 107353, 2021.

\bibitem{Cheng9340542}
P.~Cheng, S.~He, V.~Stojanovic, X.~Luan, and F.~Liu, ``Fuzzy fault detection
  for {M}arkov jump systems with partly accessible hidden information: An
  event-triggered approach,'' \emph{IEEE Transactions on Cybernetics}, pp.
  1--10, 2021.

\bibitem{Stojanovic}
V.~Stojanovic, S.~He, and B.~Zhang, ``State and parameter joint estimation of
  linear stochastic systems in presence of faults and non-{G}aussian noises,''
  \emph{International Journal of Robust and Nonlinear Control}, vol.~30,
  no.~16, pp. 6683--6700, 2020.

\bibitem{Tzortzis9157997}
I.~Tzortzis, C.~D. Charalambous, and C.~N. Hadjicostis, ``Jump {LQR} systems
  with unknown transition probabilities,'' \emph{IEEE Transactions on Automatic
  Control}, vol.~66, no.~6, pp. 2693--2708, 2021.

\bibitem{Isermann:2006}
R.~Isermann, \emph{Fault-Diagnosis Systems: An Introduction from Fault
  Detection to Fault Tolerance}.\hskip 1em plus 0.5em minus 0.4em\relax
  Springer Verlag, 2006.

\bibitem{dunford}
N.~Dunford and J.~Schwartz, \emph{Linear operators: General theory}.\hskip 1em
  plus 0.5em minus 0.4em\relax New York: Interscience Publishers, Inc., 1958.

\bibitem{pinsker}
M.~Pinsker, \emph{Information and Information Stability of Random Variables and
  Processes}.\hskip 1em plus 0.5em minus 0.4em\relax San Francisco: Holden-Day,
  Inc., 1964.

\bibitem{doi:10.1287/mnsc.1120.1641}
A.~Ben-Tal, D.~Den~Hertog, A.~De~Waegenaere, B.~Melenberg, and G.~Rennen,
  ``Robust solutions of optimization problems affected by uncertain
  probabilities,'' \emph{Management Science}, vol.~59, no.~2, pp. 341--357,
  2013.

\bibitem{MohajerinEsfahani2018}
P.~Mohajerin~Esfahani and D.~Kuhn, ``Data-driven distributionally robust
  optimization using the {W}asserstein metric: performance guarantees and
  tractable reformulations,'' \emph{Mathematical Programming}, vol. 171, no.~1,
  pp. 115--166, Sep 2018.

\bibitem{gibbs}
A.~Gibbs and F.~Su, ``On choosing and bounding probability metrics,''
  \emph{Internat. Statist. Rev}, vol.~70, no.~3, pp. 419--435, Dec. 2002.

\bibitem{ctlthem:2013}
C.~Charalambous, I.~Tzortzis, S.~Loyka, and T.~Charalambous, ``Extremum
  problems with total variation distance and their applications,'' \emph{IEEE
  Trans. Autom. Control}, vol.~59, no.~9, pp. 2353--2368, 2014.

\bibitem{Theodoridis2008}
S.~Theodoridis and K.~Koutroumbas, \emph{Pattern Recognition, Fourth
  Edition}.\hskip 1em plus 0.5em minus 0.4em\relax Academic Press, Inc., 2008.

\bibitem{Seaman06}
J.~W. Seaman~Jr. and P.~L. Odell, \emph{Variance, Upper Bounds}.\hskip 1em plus
  0.5em minus 0.4em\relax American Cancer Society, 2006.

\bibitem{Seliger:1999}
B.~Koppen-Seliger and S.~Ding, ``Fault detection and isolation of a three tank
  benchmark,'' in \emph{Proc. of the European Control Conference}, 1999.

\end{thebibliography}

\end{document}